\documentclass{article}
\usepackage{amsfonts}
\usepackage{amsmath,amsthm,amssymb,amscd}
\usepackage{latexsym}
\usepackage{graphicx}
\usepackage{url}

\makeatletter \@addtoreset{equation}{section} \makeatother

\newtheorem{theorem}{Theorem}[section]
\newtheorem{definition}{Definition}[section]

\newtheorem{lemma}{Lemma}[section]

\normalsize

\begin{document}
\title{Renormalized solutions for the fractional $p(x)$-Laplacian equation with $L^1$ data}
\author{Chao Zhang\\
\small Department of Mathematics and Institute for Advanced Study in Mathematics\\
\small Harbin Institute of Technology,
Harbin 150001, China\\
Xia Zhang\footnote{Corresponding author. E-mail addresses: piecesummer1984@163.com, czhangmath@hit.edu.cn.}\\
\small Department of Mathematics, Harbin Institute of Technology,
Harbin 150001, China\\
}

\date{ }
\maketitle

\begin{abstract}
{In this paper, we prove the existence and uniqueness of nonnegative renormalized solutions for the fractional $p(x)$-Laplacian problem with $L^1$ data. Our results are new even in the constant exponent fractional $p$-Laplacian equation case.
}

\medskip


\end{abstract}

\section{Introduction and main result}

\quad Let $\Omega\subset\mathbb{R}^{N}$ be a smooth bounded domain. In this paper, we consider the following nonlocal fractional $p(x)$-Laplacian
 equation:
\begin{eqnarray}\label{1}
\begin{cases}
\mathcal{L}u(x)=f(x),\quad &\textmd{in } \Omega,\\
u=0, &\textmd{on } \partial\Omega.
\end{cases}
\end{eqnarray}
Here we assume that
\begin{equation}
\label{assum}
0\leq f\in L^{1}(\Omega).
\end{equation}
The operator $\mathcal{L}$ is given by
\begin{equation*}
\mathcal{L}u(x):=(-\Delta)_{p(\cdot)}^s u(x)={\rm P.V.}\int_{\Omega}\frac{|u(x)-u(y)|^{p(x,y)-2}(u(x)-u(y))}{|x-y|^{N+sp(x,y)}}\,dy, \quad x\in\Omega,
\end{equation*}
where P.V. is a commonly used abbreviation in the principal value sense, $0<s<1$, $p: \overline \Omega\times \overline \Omega \to (1,\infty)$ is a continuous functions with $sp(x,y)<N$ for any $(x,y)\in\overline{\Omega}$.
This operator was first introduced by Kaufmann, Rossi and Vidal in \cite{KRV}, in which they established a compact embedding theorem and proved the existence and uniqueness of weak solutions for the fractional $p(x)$-Laplacian problem
\begin{eqnarray}
\begin{cases}
\mathcal{L} u(x)+|u(x)|^{q(x)-2}u(x)=f(x),\quad &\textmd{in } \Omega,\\
u=0, &\textmd{on } \partial\Omega,
\end{cases}
\end{eqnarray}
 provided $f\in L^{a(x)}(\Omega)$ for some $a(x)>1$.

 In the constant exponent case, the operator $\mathcal{L}$ is known as the regional fractional $p$-Laplacian, see \cite{DS}. The regional fractional Laplacian arises, for instance, from the Feller generator of the reflected symmetric stable process (\cite{BBC,CK,G,GM}). On the other hand, this operator is also  a fractional version of the $p(x)$-Laplacian, given by $\mbox{div}(|\nabla u|^{p(x)-2}\nabla u)$, which is associated with the variable exponent Sobolev space.

Regarding the non-local  $p$-Laplacian operator $(-\Delta)_p^s$, the linear elliptic case $p=2$ has been studied in \cite{AAB, KPU, LPPS}. In particular, the existence and uniqueness of renormalized solutions for the problems of the kind
$$
\beta(u)+(-\Delta)^{s}u\ni f \quad\textmd{in } \mathbb R^N
$$
was proved by Alibaud, Andreianov and Bendahmane in \cite{AAB2}, where $f\in L^1(\mathbb R^N)$ and $\beta$ is a maximal monotone graph in $\mathbb R$. Using a duality argument, in the sense of Stampacchia, Kenneth, Petitta and Ulusoy in \cite{KPU} proved the existence and uniqueness of solutions to non-local problems like $
(-\Delta)^{s}u=\mu$ in  $\mathbb R^N$ with $\mu$ being a bounded Radon measure whose support is compactly contained in $\mathbb R^N$. In \cite{KMS}, Kuusi, Mingione and Sire discussed the elliptic non-local case $p\not=2$ with measure data and developed an existence of SOLA, regularity and Wolf potential theory.  In addition, Abdellaoui \emph{et al} in \cite{AAB} investigated the fractional elliptic $p$-Laplacian equations with weight and general datum and showed that there exists a unique nonnegative entropy solution.

In this paper, we focus our attention on the existence and uniqueness of renormalized solutions for the fractional $p(x)$-Laplacian problem (\ref{1}). It is well-known that the notion of renormalized solutions was first introduced by DiPerna and Lions \cite{DL} in their study of the Boltzmann equation. Our results can be seen as a continuation of the paper \cite{KRV} and are new even in the constant exponent fractional $p$-Laplacian equation case. We construct an approximate solution sequence and
establish some \emph{a priori} estimates in order to draw a subsequence to obtain a limit function. Then based on the strong convergence of the truncations of approximate solutions and the decomposition for the region of integration according to the different contributions, we prove that this function is a renormalized solution. Moreover,  the uniqueness of renormalized solutions follows by choosing suitable test functions.

We denote $u\in\mathcal{T}^{s,p(x,y)}_{0}(\Omega)$ if $u:\Omega\rightarrow\mathbb{R}$ is measurable and $T_{k}(u)\in W_{0}^{s,p(x,y)}(\Omega)$ for any $k>0$ (see Section 2), where the truncation function $T_{k}$ is defined by
$$T_{k}(t)=\max\{-k,\min\{k,t\}\},$$  for any $t\in\mathbb{R}$.

Next we give the definition of renormalized solutions to problem (\ref{1}) which is influenced by \cite{AAB2} and \cite{ZZ}.

\begin{definition}\label{def}
We say that
$u\in\mathcal{T}^{s,p(x,y)}_{0}(\Omega)$ is a renormalized solution to \textup{(\ref{1})} if the following conditions are satisfied:

\begin{enumerate}
  \item [{\rm (i)}] $$\lim_{h\to \infty}\int_{\{(x,y)\in\Omega\times\Omega:(u(x),u(y))\in R_{h}\}}\frac{|u(x)-u(y)|^{p(x,y)-1}}{|x-y|^{N+sp(x,y)}}\,dxdy=0,$$
where $$R_{h}=\{(v,w)\in \mathbb R^{2}: \max\{|v|,|w|\}\geq h+1\ \textmd{\rm and}\  (\min\{|v|,|w|\}\leq h \\
 \ \textmd{\rm or} \ vw<0)\};$$
  \item [{\rm(ii)}] For any $\varphi\in C_{0}^{\infty}(\Omega)$ and $S\in W^{1,\infty}(\mathbb{R})$  with compact support,
\begin{eqnarray}\label{7}
\begin{split}
\int_{\Omega\times\Omega}\frac{|u(x)-u(y)|^{p(x,y)-2}(u(x)-u(y))[(S(u)\varphi)(x)-(S(u)\varphi)(y)]}
{|x-y|^{N+sp(x,y)}}\,dxdy
\\=\int_{\Omega}f S(u)\varphi\,dx
\end{split}
\end{eqnarray}
holds.
\end{enumerate}
\end{definition}

The main result of this work is the following theorem:

 \begin{theorem}
 \label{main theorem}
 Under the integrability condition \eqref{assum}, there exists a unique nonnegative renormalized solution to problem \eqref{1}.
 \end{theorem}

The rest of this paper is organized as follows. In Section 2, we collect some basic properties for variable exponent Sobolev fractional spaces which will be used later. We will prove the main result in Section 3.

\section{Preliminaries}

\quad For the convenience of the readers, we recall some definitions and basic properties
 of variable exponent Sobolev fractional spaces.
 For a deeper treatment on these spaces, we refer to
\cite{DR} and \cite{KRV}.
For a smooth bounded domain $\Omega\subset\mathbb{R}^{N}$,  let
$$p:\overline{\Omega}\times\overline{\Omega}\rightarrow(1,\infty)$$
 and
 $$q:\overline{\Omega}\rightarrow(1,\infty)$$ be two continuous functions.  We assume that $p$ is symmetric, i.e. $p(x,y)=p(y,x)$ and
 $$1<p_{-}=\inf_{(x,y)\in\overline{\Omega}\times\overline{\Omega}}p(x,y)\leq p_{+}=\sup_{(x,y)\in\overline{\Omega}\times\overline{\Omega}}p(x,y)<\infty,$$
  and
  $$1<q_{-}=\inf_{x\in\overline{\Omega}}q(x)\leq q_{+}=\sup_{x\in\overline{\Omega}}q(x)<\infty.$$

For $0<s<1$, the variable exponent Sobolev fractional space $W^{s,q(x),p(x,y)}(\Omega)$ is the class of all functions $u\in L^{q(x)}(\Omega)$ such that
$$\int_{\Omega\times\Omega}\frac{|u(x)-u(y)|^{p(x,y)}}{t^{p(x,y)}|x-y|^{N+sp(x,y)}}\,dxdy<\infty,$$ for some $t>0$, where  $L^{q(x)}(\Omega)$ is the variable exponent Lebesgue space.

Define
$$[u]_{s,p(x,y)}(\Omega)=\inf\left\{t>0:\int_{\Omega\times\Omega}
\frac{|u(x)-u(y)|^{p(x,y)}}{t^{p(x,y)}|x-y|^{N+sp(x,y)}}\,dxdy\leq1\right\}.$$
It is the variable exponent seminorm. For simplicity, we omit the set $\Omega$ from the notation.
We could get the following properties:
\begin{lemma}\label{L2.1}
\textup{(1)} If $1\leq[u]_{s,p(x,y)}<\infty$, then $$([u]_{s,p(x,y)})^{p_{-}}\leq\int_{\Omega\times\Omega}
\frac{|u(x)-u(y)|^{p(x,y)}}{|x-y|^{N+sp(x,y)}}\,dxdy\leq([u]_{s,p(x,y)})^{p_{+}};$$

\textup{(2)} If $[u]_{s,p(x,y)}\leq 1$,  then $$([u]_{s,p(x,y)})^{p_{+}}\leq\int_{\Omega\times\Omega}
\frac{|u(x)-u(y)|^{p(x,y)}}{|x-y|^{N+sp(x,y)}}\,dxdy\leq([u]_{s,p(x,y)})^{p_{-}}.$$
\end{lemma}

\noindent \textbf{Remark 2.1.} Similarly to the discussion of the norm in variable exponent space, we could get the above results.  Here we omit the proof of Lemma \ref{L2.1}.

The space $W^{s,q(x),p(x,y)}(\Omega)$ is a Banach space with the norm
$$\|u\|_{W^{s,q(x),p(x,y)}(\Omega)}=\|u\|_{L^{q(x)}(\Omega)}+[u]_{s,p(x,y)}.$$
By $W_0^{s,q(x),p(x,y)}(\Omega)$ we denote the subspace of $W^{s,q(x),p(x,y)}(\Omega)$ which is the closure of
compactly supported functions in $\Omega$ with respect to the norm $\|\cdot\|_{W^{s,q(x),p(x,y)}(\Omega)}$.
 Especially, if $q(x)=\bar p(x):=p(x,x)$, we denote $W^{s,q(x),p(x,y)}(\Omega)$ and $W_0^{s,q(x),p(x,y)}(\Omega)$  by $W^{s,p(x,y)}(\Omega)$ and $W_0^{s,p(x,y)}(\Omega)$ (see \cite{DR}), respectivly.

For any $u\in W^{s,q(x),p(x,y)}(\Omega)$, define
\begin{equation}\label{12}
\rho(u)=\int_{\Omega\times\Omega}
\frac{|u(x)-u(y)|^{p(x,y)}}{|x-y|^{N+sp(x,y)}}\,dxdy+\int_{\Omega}|u|^{q(x)}\,dx
\end{equation}
and
\begin{equation}
\|u\|_{\rho}=\inf\left\{\lambda>0:\rho\left(\frac{u}{\lambda}\right)\leq 1\right\}.
\end{equation}
It is easy to see that $\|\cdot\|_{\rho}$ is a norm which is equivalent to the norm $\|\cdot\|_{W^{s,q(x),p(x,y)}(\Omega)}$.

\begin{lemma} {\rm ($W^{s,q(x),p(x,y)}(\Omega), \|\cdot\|_{\rho}$)} is uniformly convex and $W^{s,q(x),p(x,y)}(\Omega)$ is a reflexive Banach space.
\end{lemma}
\begin{proof} The result is essentially known. Here is a short proof of it. As $p_{+}<\infty$, from the definition of $\rho$ we know that $\rho$ satisfies $\Delta_2$-condition, i.e. there exists $K\geq2$ such that
 $$\rho(2u)\leq K\rho(u)$$ for all $u\in W^{s,q(x),p(x,y)}(\Omega)$.

Since $p_{-}>1$,  similar to the proof of Theorem 3.4.9 in \cite{DHHR}, we could verify that $\rho$ is a uniformly convex semimodular, i.e. for any $\varepsilon>0$, there exists $\delta>0$ such that
$$\rho\left(\frac{u-v}{2}\right)\leq\varepsilon\frac{\rho(u)+\rho(v)}{2}$$
or
$$\rho\left(\frac{u+v}{2}\right)\leq(1-\delta)\frac{\rho(u)+\rho(v)}{2}$$
for all $u,v\in W^{s,q(x),p(x,y)}(\Omega)$.

Theorem 2.4.14 in \cite{DHHR} further implies that the norm $\|\cdot\|_{\rho}$ is uniformly convex and ($W^{s,q(x),p(x,y)}(\Omega), \|\cdot\|_{\rho}$) is uniformly convex.
Hence, $W^{s,q(x),p(x,y)}(\Omega)$ is a reflexive Banach space by virtue of Theorem 1.20 in \cite{A}.
\end{proof}


In the following, we give a compact embedding theorem into the variable exponent  Lebesgue spaces.
\begin{lemma}\label{L2.3}
Let $\Omega\subset\mathbb{R}^{N}$ be a smooth bounded domain and $s\in(0,1)$.
 Let $q(x)$, $p(x,y)$ be continuous variable exponents with $sp(x,y)<N$
 for $(x,y)\in\overline{\Omega}\times\overline{\Omega}$ and $q(x)\geq p(x,x)$ for $x\in\overline{\Omega}$. Assume that $r:\overline{\Omega}\rightarrow(1,\infty)$
is a continuous function such that
$$p^{*}(x):=\frac{Np(x,x)}{N-sp(x,x)}>r(x)\geq r_{-}>1,$$
for $x\in\overline{\Omega}$. Then, there exists a constant $C=C(N,s,p,q,r,\Omega)$ such that for every $u\in W^{s,q(x),p(x,y)}(\Omega)$, it holds that
$$\|u\|_{L^{r(x)}(\Omega)}\leq C\|u\|_{W^{s,q(x),p(x,y)}(\Omega)}.$$
That is, the space $W^{s,q(x),p(x,y)}(\Omega)$ is continuously embedded in $L^{r(x)}(\Omega)$. Moreover, this embedding is compact.

In addition, if $u\in W_0^{s,q(x),p(x,y)}(\Omega)$, it holds that
$$\|u\|_{L^{r(x)}(\Omega)}\leq C[u]_{s,p(x,y)}.$$
\end{lemma}

\noindent\textbf{Remark 2.2.} (1) We would like to mention that the compact embedding theorem has been proved in \cite{KRV} under the assumption $q(x)>p(x,x)$. Here we give a slightly different version of compact embedding  theorem assuming that $q(x)\geq p(x,x)$ which can be obtained by following the same discussions in \cite{KRV}.

(2)  Since
$\frac{Np(x,x)}{N-sp(x,x)}>\bar p(x)\geq p_{-}>1,$
Lemma \ref{L2.3} implies that $[u]_{s,p(x,y)}$ is a norm on $W_0^{s,p(x,y)}(\Omega)$, which is equivalent to the norm $\|\cdot\|_{W^{s,p(x,y)}(\Omega)}$.

\section{Proof of the main result}

\quad  In order to discuss Eq. (\ref{1}), we  restrict ourselves to $sp_->1$ to
 have a well defined trace on $\partial\Omega$. In fact, there exist $\widetilde{s}\in(0,s)$ and $r\in(0,p_{-})$ such that $\widetilde{s}r\in(0,N)$. Therefore, $W^{s,p(x,y)}(\Omega)$ is continuously embedded in $L^{q}(\partial\Omega)$ for all $q\in[1, \frac{(N-1)r}{N-\widetilde{s}r}]$ (see \cite{DR}). That is, for any $u\in W^{s,p(x,y)}(\Omega)$, $u|_{\partial\Omega}$ is well defined.

Now we are ready to prove the main results. Some of the reasoning is
based on the ideas developed in \cite{AAB, ZZ}.

 We first introduce the approximate problems. Define $T_{n}(f)=f_{n}$,  we know that $0\leq f_{n}\leq f$ such that
$$f_{n}\rightarrow f  \quad \textup{strongly in}\ L^1(\Omega).$$
Consider the  following approximate problem of (\ref{1})
\begin{eqnarray}\label{2}
\begin{cases}
\mathcal{L}u=f_{n}(x),\quad &\textmd{in } \Omega, \\
u=0, &\textmd{on } \partial\Omega.
\end{cases}
\end{eqnarray}

\begin{lemma}\label{L1}For   any $n\in\mathbb{N}$,
 there exists a unique weak solution $u_{n}\in W_0^{s,p(x,y)}(\Omega)$ to \textup{(\ref{2})} in the sense that for any $v\in W_0^{s,p(x,y)}(\Omega)$,
 \begin{equation*}
\int_{\Omega\times\Omega}\frac{|u_{n}(x)-u_{n}(y)|^{p(x,y)-2}(u_{n}(x)-u_{n}(y))(v(x)-v(y))}{|x-y|^{N+sp(x,y)}}dxdy
=\int_{\Omega}f_{n}v\,dx.
\end{equation*}
Besides, $\{u_n\}_{n}$ is an increasing nonnegative sequence.
\end{lemma}
\begin{proof} For any $u\in W_0^{s,p(x,y)}(\Omega)$,
define $$F(u)=\int_{\Omega\times\Omega}\frac{|u(x)-u(y)|^{p(x,y)}}{|x-y|^{N+sp(x,y)}}\,dxdy
-\int_{\Omega}f_{n}u\,dx.$$
Then, for any $u\in W_0^{s,p(x,y)}(\Omega)$ with $[u]_{s,p(x,y)}\geq1$, by using Lemmas \ref{L2.1} and \ref{L2.3} we derive that
\begin{displaymath}
\begin{split}
F(u)& \geq\frac{1}{p_{+}}\int_{\Omega\times\Omega}\frac{|u(x)-u(y)|^{p(x,y)}}{|x-y|^{N+sp(x,y)}}\,dxdy
-\|f_{n}\|_{L^{(\bar p(x))'}(\Omega)}\|u\|_{L^{\bar p(x)}(\Omega)}\\
&\geq\frac{1}{p_{+}}([u]_{s,p(x,y)})^{p_{-}}-C[u]_{s,p(x,y)},
\end{split}
\end{displaymath}
which implies that
$F$ is  coercive  on $W_0^{s,p(x,y)}(\Omega)$. Then,  there is a unique minimizer $u_{n}$ of $F$.
Similar to the proof of Theorem 1.4 in \cite{KRV}, we could also verify that $u_n$ is a weak solution to problem (\ref{1}).

Thanks to $f\geq0$ and $f_n=T_n(f)$, we get that $\{u_n\}_{n}$ is an increasing nonnegative sequence.
\end{proof}

Let $u:\Omega\rightarrow\mathbb{R}$ be a measure function. In the following, for simplicity, we denote
$$U_{n}(x,y)=|u_{n}(x)-u_{n}(y)|^{p(x,y)-2}(u_{n}(x)-u_{n}(y)),$$
 $$\{u>t\}=\{x\in\Omega:u(x)>t\}, \quad \{u\leq t\}=\{x\in\Omega:u(x)\leq t\},$$
and denote $|E|$ by the Lebesgue measure of a measurable set $E$ and
$$
d\nu=\frac{dxdy}{|x-y|^{N+sp(x,y)}}.
$$

\begin{lemma}\label{L2}
There  exists $u\in\mathcal{T}^{s,p(x,y)}_{0}(\Omega)$ such that $u_{n}\rightarrow u$ in measure and $u_{n}\rightarrow u$ a.e. in $\Omega$, as $n\rightarrow\infty$.
\end{lemma}
\begin{proof}
Taking $T_{k}(u_{n})$ as a  test  function  in  (\ref{2}), we get
\begin{eqnarray}\label{3}
\begin{split}
&\int_{\Omega\times\Omega}U_{n}(x,y)
[T_{k}(u_{n})(x)-T_{k}(u_{n})(y)]\,d\nu=\int_{\Omega}f_{n}T_{k}(u_{n})\,dx\leq k\int_{\Omega}f_{n}\,dx\leq Ck,
\end{split}
\end{eqnarray}
where  $C$ is  independent  of  $k$ and $n$.

Note that
\begin{eqnarray*}
\begin{split}
&|T_{k}(u_{n})(x)-T_{k}(u_{n})(y)|^{p(x,y)}\leq U_{n}(x,y)
[T_{k}(u_{n})(x)-T_{k}(u_{n})(y)],
\end{split}
\end{eqnarray*}
from (\ref{3}) we obtain
\begin{equation}\label{14}
\int_{\Omega\times\Omega}|T_{k}(u_{n})(x)-T_{k}(u_{n})(y)|^{p(x,y)}d\nu\leq Ck.
\end{equation}
Then,
$\{T_{k}(u_{n})\}_{n}$  is  bounded  in  $W_0^{s,p(x,y)}(\Omega)$.
In the following,  we assume that
$$T_{k}(u_{n})\rightharpoonup v\ \ \textup{weakly\ in}\  W_0^{s,p(x,y)}(\Omega),$$ as $n\rightarrow\infty$.
As $W_0^{s,p(x,y)}(\Omega)\hookrightarrow L^{\bar p(x)}(\Omega)$ is  compact due to Lemma 2.3, $T_{k}(u_{n})\rightarrow v$  strongly in $L^{\bar p(x)}(\Omega)$. Passing to a subsequence, still denoted by $\{u_{n}\}_{n}$, we assume that
$$T_{k}(u_{n})\to v\quad \textup{a.e.\ in}\ \Omega.$$

From (\ref{14}), for any $k\geq 1$, we have
$$\int_{\Omega\times\Omega}\left|\frac{T_{k}(u_{n})(x)-T_{k}(u_{n})(y)}{k^{\frac{1}{p_{-}}}}\right|^{p(x,y)}d\nu\leq C.$$
As $[u]_{s,p(x,y)}$ is a norm on   $W_0^{s,p(x,y)}$, it follows from Lemmas \ref{L2.1} and \ref{L2.3}
that
$$\|k^{-\frac{1}{p_{-}}}T_{k}(u_{n})\|_{L^{\bar{p}(x)}(\Omega)}\leq C,$$
where  $C$ is  independent  of  $k$ and $n$. Then,
\begin{displaymath}
\begin{split}
\big|\{u_{n}\geq k\}\big|&=\big|\{T_{k}(u_{n})=k\}\big|\\
&\leq\int_{\Omega}\left|\frac{T_{k}(u_{n})}{k}\right|^{\bar p(x)}\,dx\\
&\leq k^{1-p_{-}}\int_{\Omega}\left|k^{-\frac{1}{p_{-}}}T_{k}(u_{n})\right|^{\bar p(x)}\,dx\leq Ck^{1-p_{-}},
\end{split}
\end{displaymath}
which implies
\begin{equation}\label{4}
\lim_{k\rightarrow\infty}\limsup_{n\rightarrow\infty}\big|\{u_{n}\geq k\}\big|=0.
 \end{equation}

For any $t>0$, we get
\begin{eqnarray}\label{5}
\begin{split}
\big|\{|u_{n}-u_{m}|>t\}\big|
\leq\big|\{u_{n}>k\}\big|+\big|\{u_{m}>k\}\big|
+\big|\{|T_{k}(u_{n})-T_{k}(u_{m})|>t\}\big|.
\end{split}
\end{eqnarray}
Note that
\begin{displaymath}
\begin{split}
&\big|\{|T_{k}(u_{n})-T_{k}(u_{m})|>t\}\big|\\
& \leq\int_{\{|T_{k}(u_{n})-T_{k}(u_{m})|>t\}}
\left|\frac{T_{k}(u_{n})-T_{k}(u_{m})}{t}\right|^{\bar p(x)}\,dx\\
& \leq(t^{-p_{-}}+t^{-p_{+}})\int_{\{|T_{k}(u_{n})-T_{k}(u_{m})|>t\}}
|T_{k}(u_{n})-T_{k}(u_{m})|^{\bar p(x)}\,dx,
\end{split}
\end{displaymath}
we have
\begin{equation}\label{6}
\lim_{m,n\rightarrow\infty}\big|\{|T_{k}(u_{n})-T_{k}(u_{m})|>t\}\big|=0.
\end{equation}
From (\ref{4})--(\ref{6}), then
$$\lim_{m,n\rightarrow\infty}\big|\{|u_{n}-u_{m}|>t\}\big|=0,$$
which implies that  $u_{n}\rightarrow u$  in  measure  and $u_{n}\rightarrow u$  a.e. in $\Omega$,  as  $n\rightarrow \infty$.

Then $v=T_{k}(u)\in W_0^{s,p(x,y)}(\Omega)$ and $T_{k}(u_{n})\rightarrow T_{k}(u)$  strongly in   $L^{\bar p(x)}(\Omega)$. As $\{u_{n}\}_{n}$ is increasing, we have $u_{n}(x)\leq u(x)$ a.e.  in  $\Omega$.
\end{proof}

\begin{lemma}\label{L3}
 For any $k>0$,
$T_{k}(u_{n})\rightarrow T_{k}(u)$ strongly in $W_0^{s,p(x,y)}(\Omega)$ as $n\rightarrow\infty$.
\end{lemma}
\begin{proof}
Taking $T_{k}(u_{n})-T_{k}(u)$   as  a  test  function in (\ref{2}) to yield that
\begin{equation}\label{11}
\langle\mathcal{L}u_{n},T_{k}(u_{n})-T_{k}(u)\rangle=\int_{\Omega}f_{n}(T_{k}(u_{n})-T_{k}(u))\,dx.
\end{equation}
Denote $$ I_{1,n}=\int_{\Omega\times\Omega}U_{n}(x,y)
[T_{k}(u_{n})(x)-T_{k}(u_{n})(y)]\,d\nu$$
and
$$I_{2,n}=\int_{\Omega\times\Omega}U_{n}(x,y)
[T_{k}(u)(x)-T_{k}(u)(y)]\,d\nu.$$
From (\ref{11}), we have
\begin{equation}\label{10}
I_{1,n}=I_{2,n}+\int_{\Omega}f_{n}(T_{k}(u_{n})-T_{k}(u))\,dx.
\end{equation}

Denote $$ T_{n,k}(x,y)=|T_{k}(u_{n})(x)-T_{k}(u_{n})(x)|^{p(x,y)-2}[T_{k}(u_{n})(x)-T_{k}(u_{n})(x)].$$
Then
\begin{displaymath}
\begin{split}
I_{1,n}=&\int_{\Omega\times\Omega}|T_{k}(u_{n})(x)-T_{k}(u_{n})(y)|^{p(x,y)}\,d\nu\\
&+\int_{\Omega\times\Omega}(U_{n}(x,y)-T_{n,k}(x,y))[T_{k}(u_{n})(x)-T_{k}(u_{n})(y)]\,d\nu
\end{split}
\end{displaymath}
and
\begin{displaymath}
\begin{split}
I_{2,n}=&\int_{\Omega\times\Omega}T_{n,k}(x,y)[T_{k}(u)(x)-T_{k}(u)(y)]\,d\nu\\
&+\int_{\Omega\times\Omega}(U_{n}(x,y)-T_{n,k}(x,y))[T_{k}(u)(x)-T_{k}(u)(y)]\,d\nu\\
\leq&\int_{\Omega\times\Omega}\frac{1}{p(x,y)}|T_{k}(u)(x)-T_{k}(u)(y)|^{p(x,y)}\,d\nu\\
&+\int_{\Omega\times\Omega}\frac{p(x,y)-1}{p(x,y)}|T_{n,k}(x,y)|^{\frac{p(x,y)}{p(x,y)-1}}\,d\nu\\
&+\int_{\Omega\times\Omega}(U_{n}(x,y)-T_{n,k}(x,y))[T_{k}(u)(x)-T_{k}(u)(y)]\,d\nu.\\
\end{split}
\end{displaymath}
From (\ref{10}), we have
\begin{displaymath}
\begin{split}
&\int_{\Omega\times\Omega}
\frac{|T_{k}(u_{n})(x)-T_{k}(u_{n})(y)|^{p(x,y)}}{p(x,y)}\,d\nu\\
&\quad+\int_{\Omega\times\Omega}(U_{n}(x,y)-T_{n,k}(x,y))[T_{k}(u_{n})(x)-T_{k}(u)(x)-T_{k}(u_{n})(y)+T_{k}(u)(y)]\,d\nu\\
&\leq\int_{\Omega\times\Omega}
\frac{|T_{k}(u)(x)-T_{k}(u)(y)|^{p(x,y)}}{p(x,y)}d\nu+\int_{\Omega}f_{n}(T_{k}(u_{n})-T_{k}(u))\,dx.
\end{split}
\end{displaymath}

 In the following, we will verify that the second term on the left-hand side of the above inequality  is nonnegative. We divide $\Omega\times\Omega$ into the following four parts:
$$A_{1}=\{(x,y)\in\Omega\times\Omega:u_{n}(x)\leq k,u_{n}(y)\leq k\},$$
$$A_{2}=\{(x,y)\in\Omega\times\Omega:u_{n}(x)\geq k,u_{n}(y)\geq k\},$$
$$A_{3}=\{(x,y)\in\Omega\times\Omega:u_{n}(x)\leq k,u_{n}(y)\geq k\},$$
$$A_{4}=\{(x,y)\in\Omega\times\Omega:u_{n}(x)\geq k,u_{n}(y)\leq k\}.$$
Similar to the proof of Lemma 3.6 in \cite{AAB}, we could verifty that
$$(U_{n}(x,y)-T_{n,k}(x,y))[T_{k}(u_{n})(x)-T_{k}(u)(x)-T_{k}(u_{n})(y)+T_{k}(u)(y)]\geq0$$
 a.e. in $A_{1}\cup A_{2}\cup A_{3}\cup A_{4}.$
Then
$$\int_{\Omega\times\Omega}(U_{n}(x,y)-T_{n,k}(x,y))
[T_{k}(u_{n})(x)-T_{k}(u)(x)-T_{k}(u_{n})(y)+T_{k}(u)(y)]\,d\nu\geq0,$$
which implies
\begin{eqnarray}\label{13}
\begin{split}
&\int_{\Omega\times\Omega}
\frac{|T_{k}(u_{n})(x)-T_{k}(u_{n})(y)|^{p(x,y)}}{p(x,y)}\,d\nu\\
&\leq\int_{\Omega\times\Omega}\frac{|T_{k}(u)(x)-T_{k}(u)(y)|^{p(x,y)}}{p(x,y)}\,d\nu
+\int_{\Omega}f_{n}(T_{k}(u_{n})-T_{k}(u))\,dx.
\end{split}
\end{eqnarray}

As  $T_{k}(u_{n})\rightarrow T_{k}(u)$ strongly in $L^{\bar{p}(x)}(\Omega)$, we derive that as $n\rightarrow\infty$,
$$\int_{\Omega}f_{n}(T_{k}(u_{n})-T_{k}(u))\,dx\rightarrow 0.$$
It follows from  Fatou lemma and (\ref{13}) that
\begin{displaymath}
\begin{split}
&\int_{\Omega\times\Omega}
\frac{|T_{k}(u)(x)-T_{k}(u)(y)|^{p(x,y)}}{p(x,y)}\,d\nu\\
&=\int_{\Omega\times\Omega}\liminf_{n\rightarrow\infty}
\frac{|T_{k}(u_{n})(x)-T_{k}(u_{n})(y)|^{p(x,y)}}{p(x,y)}\,d\nu\\
&\leq\liminf_{n\rightarrow\infty}\int_{\Omega\times\Omega}
\frac{|T_{k}(u_{n})(x)-T_{k}(u_{n})(y)|^{p(x,y)}}{p(x,y)}\,d\nu\\
&\leq\limsup_{n\to \infty}\int_{\Omega\times\Omega}
\frac{|T_{k}(u_{n})(x)-T_{k}(u_{n})(y)|^{p(x,y)}}{p(x,y)}\,d\nu\\
&\leq\int_{\Omega\times\Omega}
\frac{|T_{k}(u)(x)-T_{k}(u)(y)|^{p(x,y)}}{p(x,y)}\,d\nu,
\end{split}
\end{displaymath}
which yields
\begin{displaymath}
\begin{split}
\lim_{n\to \infty}\int_{\Omega\times\Omega}
\frac{|T_{k}(u)(x)-T_{k}(u)(y)|^{p(x,y)}}{p(x,y)}\,d\nu
=\int_{\Omega\times\Omega}
\frac{|T_{k}(u)(x)-T_{k}(u)(y)|^{p(x,y)}}{p(x,y)}\,d\nu.
\end{split}
\end{displaymath}
Note that
\begin{displaymath}
\begin{split}
&\big|[T_{k}(u_{n})(x)-T_{k}(u_{n})(y)]-[T_{k}(u)(x)-T_{k}(u)(y)]\big|^{p(x,y)}\\
&\leq2^{p_{+}}\left(|T_{k}(u_{n})(x)-T_{k}(u_{n})(y)|^{p(x,y)}+|T_{k}(u)(x)-T_{k}(u)(y)|^{p(x,y)}\right),
\end{split}
\end{displaymath}
then by Fatou lemma, we have
\begin{displaymath}
\begin{split}
\int_{\Omega\times\Omega}&\frac{2^{p_+ +1}|T_{k}(u)(x)-T_{k}(u)(y)|^{p(x,y)}}{p(x,y)}\,d\nu\\
=\int_{\Omega\times\Omega}&\frac{1}{p(x,y)}\liminf_{n\rightarrow\infty}\Big(2^{p_+}\big|T_{k}(u_n)(x)-T_{k}(u_n)(y)\big|^{p(x,y)}
+2^{p_+}\big|T_{k}(u)(x)-T_{k}(u)(y)\big|^{p(x,y)}\\
&-\left|[T_{k}(u_{n})(x)-T_{k}(u_{n})(y)]-[T_{k}(u)(x)-T_{k}(u)(y)]\right|^{p(x,y)}\Big)\,d\nu\\
\leq\liminf_{n\rightarrow\infty}&\int_{\Omega\times\Omega}\frac{1}{p(x,y)}\Big(2^{p_+}\big|T_{k}(u_n)(x)-T_{k}(u_n)(y)\big|^{p(x,y)}
+2^{p_+}\big|T_{k}(u)(x)-T_{k}(u)(y)\big|^{p(x,y)}\\
&-\big|[T_{k}(u_{n})(x)-T_{k}(u_{n})(y)]-[T_{k}(u)(x)-T_{k}(u)(y)]\big|^{p(x,y)}\Big)\,d\nu\\
=\int_{\Omega\times\Omega}&\frac{2^{p_+ +1}\big|T_{k}(u)(x)-T_{k}(u)(y)\big|^{p(x,y)}}{p(x,y)}\,d\nu\\
&-\limsup_{n\rightarrow\infty}\int_{\Omega\times\Omega}\frac{\big|[T_{k}(u_{n})(x)-T_{k}(u_{n})(y)]-[T_{k}(u)(x)-T_{k}(u)(y)]\big|^{p(x,y)}}{p(x,y)}\,d\nu.
\end{split}
\end{displaymath}
Thus,
$$\int_{\Omega\times\Omega}\frac{\big|[T_{k}(u_{n})(x)-T_{k}(u_{n})(y)]-[T_{k}(u)(x)-T_{k}(u)(y)]\big|^{p(x,y)}}{p(x,y)}\,d\nu\rightarrow0.$$
As $[u]_{s,p(x,y)}$ is a norm on   $W_0^{s,p(x,y)}$, it follows from Lemma \ref{L2.1}
that
$T_{k}(u_{n})\rightarrow T_{k}(u)$ strongly in $W_0^{s,p(x,y)}(\Omega)$.
\end{proof}

\begin{theorem}
 The function $u$ obtained in Lemma \ref{L2} is  a  unique renormalized  solution  to  problem \textup{(\ref{1})}.
\end{theorem}
\begin{proof} (i) Existence of renormalized solutions.
We will divide the proof into the following two steps.

\textbf{Step 1.} We will verify that
$$\lim_{h\to \infty}\int_{\{(u(x),u(y))\in R_{h}\}}|u(x)-u(y)|^{p(x,y)-1}\,d\nu=0.$$
For any $h>0$, denote $G_{h}(t)=t-T_{h}(t)$.
Taking $T_{1}(G_{h}(u_{n}))$  as  a  test  function  in  (\ref{2}), we have
\begin{displaymath}
\begin{split}
&\int_{\Omega\times\Omega}U_{n}(x,y)
[T_{1}(G_{h}(u_{n}))(x)-T_{1}(G_{h}(u_{n}))(y)]\,d\nu\\
&=\int_{\Omega}f_{n}T_{1}(G_{h}(u_{n}))\,dx\\
&\leq\int_{\{u_{n}>h\}}f_{n}\,dx\leq\int_{\{u_{n}>h\}}f\,dx.
\end{split}
\end{displaymath}
If $(u_{n}(x),u_{n}(y))\in R_{h}$,
\begin{align*}
&|u_{n}(x)-u_{n}(y)|^{p(x,y)-1}\leq U_{n}(x,y)
[T_{1}(G_{n}(u_{n}))(x)-T_{1}(G_{n}(u_{n}))(y)].
\end{align*}
Then
\begin{equation}\label{8}
\int_{\{(u_n(x),u_n(y))\in R_{h}\}}|u_{n}(x)-u_{n}(y)|^{p(x,y)-1}\,d\nu
\leq\int_{\{u_{n}>h\}}f\,dx.
\end{equation}
By Fatou lemma,
\begin{displaymath}
\begin{split}
&\int_{\Omega\times\Omega}|u(x)-u(y)|^{p(x,y)-1}
\chi_{\{(u(x),u(y))\in R_{h}\}}\,d\nu\\
&=\int_{\Omega\times\Omega}\liminf_{n\to \infty}|u_{n}(x)-u_{n}(y)|^{p(x,y)-1}
\chi_{\{(u_n(x),u_n(y))\in R_{h}\}}\,d\nu\\
&\leq\liminf_{n\to \infty}\int_{\Omega\times\Omega}|u_{n}(x)-u_{n}(y)|^{p(x,y)-1}
\chi_{\{(u_n(x),u_n(y))\in R_{h}\}}\,d\nu\\
&\leq\liminf_{n\to \infty}\int_{\{u_{n}>h\}}f\,dx.
\end{split}
\end{displaymath}
As  $f\in L^{1}(\Omega)$, by (\ref{4}) we obtain that
\begin{displaymath}
\begin{split}
&\lim_{h\to \infty}\int_{\Omega\times\Omega}|u(x)-u(y)|^{p(x,y)-1}
\chi_{\{((u(x),u(y))\in R_{h}\}}\,d\nu\\
&\leq\lim_{h\to \infty}\lim_{n\to \infty}\int_{\{u_{n}>h\}}f\,dx=0.
\end{split}
\end{displaymath}

\textbf{Step 2.}   For any $\varphi\in C_{0}^{\infty}(\Omega)$ and $S\in W^{1,\infty}(\mathbb{R})$  with compact support,
we will verify that $u$ satisfies (\ref{7}).

Taking $S(u_{n})\varphi$ as a test function in (\ref{2}), we have
\begin{eqnarray}\label{9}
\begin{split}
&\int_{\Omega\times\Omega}U_{n}(x,y)[(S(u_{n})\varphi)(x)-(S(u_{n})\varphi)(y)]\,d\nu=\int_{\Omega}f_{n} S(u_{n})\varphi\,dx.
\end{split}
\end{eqnarray}
Note that
\begin{displaymath}
\begin{split}
&\int_{\Omega\times\Omega}U_{n}(x,y)[(S(u_{n})\varphi)(x)-(S(u_{n})\varphi)(y)]\,d\nu\\
=&\int_{\Omega\times\Omega}U_{n}(x,y)
[S(u_{n})(x)-S(u_{n})(y)]\cdot\frac{\varphi(x)+\varphi(y)}{2}\,d\nu\\
&+\int_{\Omega\times\Omega}U_{n}(x,y)
(\varphi(x)-\varphi(y))\cdot\frac{S(u_{n})(x)+S(u_{n})(y)}{2}\,d\nu\\
:=& I_{1}+I_{2}.
\end{split}
\end{displaymath}

In the following, we assume that ${\rm supp}\,S\subset[-M,M]$, where $M>0$ and define the following subdomains of $\Omega\times\Omega$:
\begin{align*}
&B_{1,n}=\{(x,y)\in\Omega\times\Omega: u_{n}(x)\geq M,u_{n}(y)\geq M\},\\
&B_{2,n}=\{(x,y)\in\Omega\times\Omega: u_{n}(x)\leq M,u_{n}(y)\leq M\},\\
&B_{3,n}=\{(x,y)\in\Omega\times\Omega: M\leq u_{n}(x)\leq M+1,u_{n}(y)\leq M\},\\
&B_{4,n}=\{(x,y)\in\Omega\times\Omega: u_{n}(x)\geq M+1,u_{n}(y)\leq M\},\\
&B_{5,n}=\{(x,y)\in\Omega\times\Omega: u_{n}(x)\leq M, M\leq u_{n}(y)\leq M+1\},\\
&B_{6,n}=\{(x,y)\in\Omega\times\Omega: u_{n}(x)\leq M, u_{n}(y)\geq M+1\}.
\end{align*}
 First, we  will  estimate $I_{1}$ and denote
\begin{align*}
&G_n(x,y)=\frac{U_{n}(x,y)
[S(u_{n})(x)-S(u_{n})(y)]}
{|x-y|^{N+sp(x,y)}}\cdot\frac{\varphi(x)+\varphi(y)}{2},\\
&G(x,y)=\frac{U(x,y)
[S(u)(x)-S(u)(y)]}
{|x-y|^{N+sp(x,y)}}\cdot\frac{\varphi(x)+\varphi(y)}{2},
\end{align*}
where
$$
U(x,y)=|u(x)-u(y)|^{p(x,y)-2}(u(x)-u(y)).
$$

(1) In   $B_{1,n}$, $S(u_{n})(x)=S(u_{n})(y)=0$. Then, $G_n(x,y)=0$.

(2) In
$B_{2,n}$, $T_{M}(u_{n})(x)=u_{n}(x)$ and $T_{M}(u_{n})(y)=u_{n}(y).$
Then
\begin{displaymath}
\begin{split}
&U_{n}(x,y)
[S(u_{n})(x)-S(u_{n})(y)]\\
&=|T_{M}(u_{n})(x)-T_{M}(u_{n})(y)|^{p(x,y)-2}[T_{M}(u_{n})(x)-T_{M}(u_{n})(y)]\\
&\quad\cdot [S(T_{M}(u_{n}))(x)-S(T_{M}(u_{n}))(y)].
\end{split}
\end{displaymath}
Note that
$$
S(T_{M}(u_{n}))(x)-S(T_{M}(u_{n}))(y)
=S'(\xi)[T_{M}(u_{n})(x)-T_{M}(u_{n})(y)],
$$
where $\xi$ is between $T_{M}(u_{n})(x)$ and $T_{M}(u_{n})(y)$.
We could verify that
$$\left\{\frac{
[S(T_{M}(u_{n}))(x)-S(T_{M}(u_{n}))(y)]}
{|x-y|^{\frac{N+sp(x,y)}{p(x,y)}}}\cdot\frac{\varphi(x)+\varphi(y)}{2}\cdot
\chi_{B_{2,n}}\right\}_n$$
 is bounded in $\ L^{p(x,y)}(\Omega\times\Omega)$. Besides, it follows from Lemma \ref{L3} that
\begin{displaymath}
\begin{split}
&\frac{|T_{M}(u_{n})(x)-T_{M}(u_{n})(y)|^{p(x,y)-2}[T_{M}(u_{n})(x)-T_{M}(u_{n})(y)]}
{|x-y|^{(N+sp(x,y))\frac{p(x,y)-1}{p(x,y)}}}\\
&\quad \rightarrow\frac{|T_{M}(u)(x)-T_{M}(u)(y)|^{p(x,y)-2}[T_{M}(u)(x)-T_{M}(u)(y)]}
{|x-y|^{(N+sp(x,y))\frac{p(x,y)-1}{p(x,y)}}}\\
&\textmd{strongly in } L^{\frac{p(x,y)}{p(x,y)-1}}(\Omega\times\Omega).
\end{split}
\end{displaymath}
Then, we obtain
\begin{displaymath}
\begin{split}
\int_{B_{2,n}}&G_n(x,y)\,dxdy\\
=\int_{\Omega\times\Omega}& |T_{M}(u_n)(x)-T_{M}(u_n)(y)|^{p(x,y)-2}[T_{M}(u_n)(x)-T_{M}(u_n)(y)]
\\
&\cdot[S(u_n)(x)-S(u_n)(y)]\frac{\varphi(x)+\varphi(y)}{2}\chi_{B_{2,n}}\,d\nu\\
\rightarrow&\int_{\Omega\times\Omega} |T_{M}(u)(x)-T_{M}(u)(y)|^{p(x,y)-2}[T_{M}(u)(x)-T_{M}(u)(y)]
\\
&\cdot[S(u)(x)-S(u)(y)]\frac{\varphi(x)+\varphi(y)}{2}\chi_{\{u(x)\leq M,\ u(y)\leq M\}}\,d\nu.
\end{split}
\end{displaymath}

(3) In $B_{3,n}$, similar to the discussion of (2), we verify that
$$\lim_{n\to \infty}\int_{B_{3,n}}G_n(x,y)\,dxdy=\int_{\{M\leq u(x)\leq M+1,u(y)\leq M\}}G(x,y)\,dxdy.$$

(4) In $B_{4,n}$, we have
$$\max\{u_{n}(x),u_{n}(y)\}\geq M+1\ \textup{and}\ \min\{u_{n}(x),u_{n}(y)\}\leq M,$$
 which implies
$(u_{n}(x),u_{n}(y))\in R_{M}$. By (\ref{8}), we conclude that
$$\lim_{M\to \infty}\lim_{n\to \infty}\int_{B_{4,n}}
|u_{n}(x)-u_{n}(y)|^{p(x,y)-1}\,d\nu=0.$$
Thus
$$
\lim_{M\to \infty}\lim_{n\to \infty}\int_{B_{4,n}}
G_n(x,y)\,dxdy=0.
$$

Since
$$\int_{B_{3,n}}G_n(x,y)\,dxdy=\int_{B_{5,n}}G_n(x,y)\,dxdy$$
and $$\int_{B_{4,n}}G_n(x,y)\,dxdy=\int_{B_{6,n}}G_n(x,y)\,dxdy,$$
we have
$$I_{1}=\left(\int_{B_{1,n}}+\int_{B_{2,n}}
+2\int_{B_{3,n}}+2\int_{B_{4,n}}\right)G_n(x,y)\,dxdy.$$

It follows from  (1)--(4) that
\begin{displaymath}
\begin{split}
\lim_{n\to \infty}I_{1}=&\lim_{M\to \infty}\lim_{n\to \infty}I_{1}\\
=&\lim_{M\to \infty}\int_{\{u(x)\leq M,u(y)\leq M\}}G(x,y)\,dxdy\\
&+2\lim_{M\to \infty}\int_{\{M\leq u(x)\leq M+1,u(y)\leq M\}}G(x,y)\,dxdy\\
&+\lim_{M\to \infty}2\lim_{n\to \infty}\int_{B_{4,n}}G_n(x,y)\,dxdy\\
=&\int_{\Omega\times\Omega}G(x,y)\,dxdy.
\end{split}
\end{displaymath}
Similarly, we could verify that
$$I_{2}\rightarrow\int_{\Omega\times\Omega}U(x,y)
(\varphi(x)-\varphi(y))\cdot\frac{S(u)(x)+S(u)(y)}{2}\,d\nu.$$
Besides,
$$\int_{\Omega}f_{n} S(u_{n})\varphi \,dx\rightarrow \int_{\Omega}f S(u) \varphi\,dx.$$
Thus by (\ref{9}), we find
$$\int_{\Omega\times\Omega}\frac{|u(x)-u(y)|^{p(x,y)-2}(u(x)-u(y))}
{|x-y|^{N+sp(x,y)}}[(S(u)\varphi)(x)-(S(u)\varphi)(y)]\,dxdy=\int_{\Omega}f S(u) \varphi \,dx.$$
Combining with Step 1  and Step 2, we verify that $u$ is a renormalized solution to (\ref{1}).

 (ii) Uniqueness of renormalized solutions.

Now we  prove the uniqueness of renormalized solutions to problem
(\ref{1}) by choosing an appropriate test function motivated by
\cite {BW,BMR,ZZ}. Let $u$ and $v$ be two renormalized
solutions to problem (\ref{1}). Fix a positive number $k$. For
$\sigma>0$, let $S_\sigma$ be the function defined by
\begin{eqnarray}\label{3-19}
\left\{
\begin{array}{ll}
 \displaystyle
S_\sigma(r)=r &\hbox{ if } |r|< \sigma,\\
\displaystyle S_\sigma(r)=(\sigma+\frac 12)\mp\frac
12(r\mp(\sigma+1))^2 &\hbox{ if } \sigma\le
\pm r\le \sigma+1,\\
\displaystyle S_\sigma(r)=\pm(\sigma+\frac 12) &\hbox{ if } \pm
r>\sigma+1.
\end{array}
\right.
\end{eqnarray}
It is obvious that
\begin{eqnarray*}
\left\{
\begin{array}{ll}
\displaystyle
S'_\sigma(r)=1 &\hbox{ if } |r|< \sigma,\\[2mm]
S'_\sigma(r)=\sigma+1-|r| &\hbox{ if } \sigma\le
|r|\le \sigma+1,\\[2mm]
S'_\sigma(r)=0 &\hbox{ if } |r|>\sigma+1.
\end{array}
\right.
\end{eqnarray*}

It is easy to check  $S_\sigma\in W^{1,\infty}(\mathbb R)$ with
${\rm supp}\,S'_\sigma\subset [-\sigma-1,\sigma+1]$.
Therefore,  we may take $S=S_\sigma$ in (\ref{7}) to have
\begin{eqnarray*}
&&\int_{\Omega\times \Omega}U(x,y)(\varphi(x)-\varphi(y))\cdot\frac{S_\sigma'(u)(x)+S_\sigma'(u)(y)}{2}\,d\nu \\
&&\quad+\int_{\Omega\times \Omega}U(x,y)(S_\sigma'(u)(x)-S_\sigma'(u)(y))\cdot\frac{\varphi(x)+\varphi(y)}{2}\,d\nu\\
&&=\int_\Omega f S_\sigma'(u)\varphi\,dx
\end{eqnarray*}
and
\begin{eqnarray*}
&&\int_{\Omega\times \Omega}V(x,y)(\varphi(x)-\varphi(y))\cdot\frac{S_\sigma'(v)(x)+S_\sigma'(v)(y)}{2}\,d\nu \\
&&\quad+\int_{\Omega\times \Omega}V(x,y)(S_\sigma'(v)(x)-S_\sigma'(v)(y))\cdot\frac{\varphi(x)+\varphi(y)}{2}\,d\nu\\
&&=\int_\Omega f S_\sigma'(v)\varphi\,dx,
\end{eqnarray*}
where
$$
V(x,y)=|v(x)-v(y)|^{p(x,y)-2}(v(x)-v(y)).
$$
For every fixed $k>0$, we plug $\varphi =T_k(S_\sigma(u)-S_\sigma(v))$ as a test function
in the above equalities and subtract them to obtain that
\begin{equation}
\label{3-20} J_1+J_2=J_3,
\end{equation}
where
\begin{eqnarray*}
&&J_1=\int_{\Omega\times \Omega}\left[\frac{S_\sigma'(u)(x)+S_\sigma'(u)(y)}{2}U(x,y)-\frac{S_\sigma'(v)(x)+S_\sigma'(v)(y)}{2}V(x,y)\right]\\[2mm]
&&\qquad\qquad\qquad\cdot [T_k(S_\sigma(u)-S_\sigma(v))(x)-T_k(S_\sigma(u)-S_\sigma(v))(y)]\,d\nu,\\[2mm]
&&J_2=\int_{\Omega\times \Omega}\left[U(x,y)(S_\sigma'(u)(x)-S_\sigma'(u)(y))-V(x,y)(S_\sigma'(v)(x)-S_\sigma'(v)(y))\right]\\[2mm]
&&\qquad\qquad\qquad\cdot \frac{T_k(S_\sigma(u)-S_\sigma(v))(x)+T_k(S_\sigma(u)-S_\sigma(v))(y)}{2}\,d\nu,\\[2mm]
&&J_3=\int_\Omega
f(S'_\sigma(u)-S'_\sigma(v))T_k(S_\sigma(u)-S_\sigma(v))\,dx.
\end{eqnarray*}

We estimate $J_1$, $J_2$ and $J_3$ one by one. Writing
\begin{eqnarray*}
&&J_1=\int_{\Omega\times \Omega} (U(x,y)-V(x,y))\cdot[T_k(S_\sigma(u)-S_\sigma(v))(x)-T_k(S_\sigma(u)-S_\sigma(v))(y)]\,d\nu \\[2mm]
&&\qquad+\int_{\Omega\times \Omega} \left(1-\frac{S_\sigma'(u)(x)+S_\sigma'(u)(y)}{2}\right)U(x,y)\\[2mm]
&&\qquad\qquad\cdot[T_k(S_\sigma(u)-S_\sigma(v))(x)-T_k(S_\sigma(u)-S_\sigma(v))(y)]\,d\nu \\[2mm]
&&\qquad+\int_{\Omega\times \Omega} \left(\frac{S_\sigma'(v)(x)+S_\sigma'(v)(y)}{2}-1\right)V(x,y)\\[2mm]
&&\qquad\qquad \cdot[T_k(S_\sigma(u)-S_\sigma(v))(x)-T_k(S_\sigma(u)-S_\sigma(v))(y)]\,d\nu \\[2mm]
&&\quad:=J^1_1+J^2_1+J^3_1,
\end{eqnarray*}
and setting  $\sigma\ge k$, we have
\begin{eqnarray}\label{J11}
&&J^1_1\ge \int_{\{|u-v|\le k\}\cap\{|u|,|v|\le k\}}(U(x,y)-V(x,y))\nonumber \\[2mm]
&&\qquad\qquad\cdot[(u(x)-v(x))-(u(y)-v(y))]\,d\nu.
\end{eqnarray}
By the Lebesgue dominated convergence theorem, we conclude that
$$
J_1^2, J_1^3\to 0, \quad\textmd{as } \sigma\to +\infty.
$$
Furthermore, we have
\begin{align*}
|J_2| &\le  C\Big(\int_{\{(u(x), u(y))\in R_\sigma\}} |u(x)-u(y)|^{p(x,y)-1}\,d\nu\\
&\quad+\int_{\{(v(x), v(y))\in R_\sigma\}}|v(x)-v(y)|^{p(x,y)-1}\,d\nu\Big).
\end{align*}
From the above estimates and (i) in Definition 1.1, we obtain
$$
\lim\limits_{\sigma\to +\infty} (|J^2_1|+|J^3_1|+|J_2|)=0.
$$

Observing
\begin{eqnarray*}
f(S'_\sigma(u)-S'_\sigma(v))\to 0 \quad\mbox{strongly in } L^1(\Omega)
\end{eqnarray*}
as $\sigma\to +\infty$ and using the Lebesgue dominated convergence
theorem, we deduce that
$$
\lim\limits_{\sigma\to +\infty}|J_3|=0.
$$

Therefore, sending $\sigma\to +\infty$ in (\ref{3-20}) and recalling
(\ref{J11}), we have
\begin{eqnarray*}
\int_{\{|u|\le \frac k2,|v|\le \frac k2\}}(U(x,y)-V(x,y))
\cdot[(u(x)-v(x))-(u(y)-v(y))]\,d\nu=0,
\end{eqnarray*}
which implies $u=v$ a.e. on the set $\big\{|u|\le
\frac k2,|v|\le \frac k2\big\}$. Since $k$ is arbitrary, we conclude
that $u=v$ a.e. in $\Omega_T$. This finishes the proof of Theorem \ref{main theorem}.
\end{proof}

\section*{Acknowledgements}

This work was supported by the NSFC (Nos. 11671111, 11601103) and Heilongjiang Province Postdoctoral
Startup Foundation (LBH-Q16082).

\end{document}